\documentclass[reqno]{amsart} \usepackage[centertags]{amsmath}
\usepackage{graphicx} \usepackage{amsfonts} \usepackage{amssymb}
\usepackage{amsthm} \usepackage[top=1in, bottom=1in, left=1in,
right=1in]{geometry}

\usepackage[pdfborder={0 0 0}, pdftitle={ALL FINITE SUBDIVISION RULES ARE COMBINATORIALLY EQUIVALENT TO THREE-DIMENSIONAL SUBDIVISION RULES}, pdfauthor={BRIAN RUSHTON}, pdftex, bookmarks=true,
bookmarksnumbered = true]{hyperref}

\theoremstyle{plain} \newtheorem{thm}{Theorem} \theoremstyle{plain}
\newtheorem{lemma}[thm]{Lemma} \theoremstyle{definition}
\newtheorem*{defi}{Definition} \theoremstyle{remark}
 \theoremstyle{plain}

\newcommand{\parlengths}{\setlength{\parindent}{0pt}}
\begin{document} \date{\today}

\pdfbookmark[1]{ALL FINITE SUBDIVISION RULES ARE COMBINATORIALLY EQUIVALENT TO THREE-DIMENSIONAL SUBDIVISION RULES}{user-title-page}

\title{All finite subdivision rules are combinatorially equivalent to three-dimensional subdivision rules}

\author{Brian Rushton} \address{Department of Mathematics, Brigham Young
University, Provo, UT 84602, USA}
\email{brirush@mathematics.byu.edu}

\begin{abstract} Finite subdivision rules in high dimensions can be difficult to visualize and require complex topological structures to be constructed explicitly. In many applications, only the history graph is needed. We characterize the history graph of a subdivision rule, and define a combinatorial subdivision rule based on such graphs. We use this to show that a finite subdivision rule of arbitrary dimension is combinatorially equivalent to a three-dimensional subdivision rule. We use this to show that the Gromov boundary of special cubulated hyperbolic groups is a quotient of a compact subset of three-dimensional space, with connected preimages at each point.
\end{abstract}

\maketitle\parlengths

\section{Introduction}

Finite subdivision rules are a construction in geometric group theory originally described by Cannon, Floyd, and Parry in relation to Cannon's Conjecture \cite{subdivision}. A finite subdivision rule is a rule for replacing polygons in a tiling with a more refined tiling of polygons using finitely many combinatorial rules. An example of a subdivision rule is 2-dimensional barycentric subdivision, which replaces every triangle in a two-dimensional simplicial complex with six smaller triangles. Subdivision rules have also been used to study rational maps \cite{rational,cannon2012nearly}.

Cannon and Swenson and shown that two-dimensional subdivision rules for a three-dimensional hyperbolic manifold group contain enough information to reconstruct the group itself \cite{conformal, hyperbolic}. This was later generalized to show that many groups can be associated to a subdivision rule of some dimension, and that these subdivision rules \cite{CubeSubs, polysubs}:
\begin{enumerate}
\item captures all of the quasi-isometry information of the group via a graph called the \textbf{history graph}, and
\item has simple combinatorial tests for many quasi-isometry properties \cite{SubClass}.
\end{enumerate}

There are two difficulties in using these results:
\begin{enumerate}
\item These higher-dimensional subdivision rules are often difficult to visualize or to gain intuition for.

\item Subdivision rules often contain more information than is necessary; for instance, there are many classes of subdivision rules that have the same combinatorial structures but different topological structures (see weak equivalence classes of subdivision rules \cite{subdivision}).
\end{enumerate}

In this paper, we provide a solution to both of these problems by defining a \textbf{combinatorial subdivision rule}, which is a graph with a few simple combinatorial properties. We prove the following:

\newtheorem*{SubsComb}{Theorem \ref{SubsComb}}
\begin{SubsComb}
Let $(R,X)$ be a finite subdivision pair. Then the history graph $\Gamma(R,X)$ is a combinatorial subdivision rule.
\end{SubsComb}

\newtheorem*{CombSubs}{Theorem \ref{CombSubs}}
\begin{CombSubs}
Let $\Xi$ be a combinatorial subdivision rule. Then there is a 3-dimensional subdivision pair $(R,X)$ such that the history graph $\Gamma(R,X)$ is graph isomorphic to $\Xi$.
\end{CombSubs}

These two theorems show that we can replace a finite subdivison rule of dimension $n$ with a 3-dimensional finite subdivision rule without changing its quasi-isometry properties. Also, to find a finite subdivision rule associated to a group, it is now only necessary to find a combinatorial subdivision rule quasi-isometric to the group.

\section{Definitions}
\subsection{Finite subdivision rules}\label{formalsection}

An \textbf{almost polyhedral complex of dimension} $n$ is a finite $n$-dimensional CW complex $Z$ with a fixed cell structure such that $Z$ is the union of its closed $n$-cells, and, for every closed $n$-cell $\tilde{s}$ of $Z$, there is a CW structure $s$
on a closed $n$-disk such that the subcells of $s$ are contained in
$\partial s$ and the characteristic map $\psi_s:s\rightarrow S_R$ which
maps onto $\tilde{s}$ restricts to a homeomorphism onto each open cell.

A \textbf{(colored) finite subdivision rule $R$ of dimension $n$} consists of:
\begin{enumerate}
\item An almost polyhedral complex $S_R$ of dimension $n$.
\item A finite $n$-dimensional complex $R(S_R)$ that is a subdivision of $S_R$.
\item A \textbf{coloring} of the cells of $S_R$, which is a partition of the set of cells of $S_R$ into an ideal set $I$ and a non-ideal set $N$ so that the union of the cells in $I$ is closed.
\item A \textbf{subdivision map} $\phi_R: R(S_R)\rightarrow S_R$, which is a
continuous cellular map that restricts to a homeomorphism on each open
cell, and which maps the union of all cells of $I$ into itself. \end{enumerate}

Each cell in the complex $S_R$ (with its appropriate
characteristic map) is called a \textbf{tile type} of $S_R$.

Given a finite subdivision rule $R$ of dimension $n$, an $R$-\textbf{complex}
consists of an $n$-dimensional CW complex $X$ which is the union of its
closed $n$-cells, together with a continuous cellular map $f:X\rightarrow
S_R$ whose restriction to each open cell is a homeomorphism. This map $f$ is called the \textbf{structure map} of $X$. All tile types with their characteristic maps are $R$-complexes.

We now describe how to subdivide an $R$-complex $X$ with structure map
$f:X\rightarrow S_R$, as described above. Recall that $R(S_R)$ is a
subdivision of $S_R$. By considering $f$ as a function from $X$ to $R(S_R)$, we can pull back the cell structure on $R(S_R)$
to the cells of $X$ to create $R(X)$, a subdivision of $X$. This gives
an induced map $f:R(X)\rightarrow R(S)$ that restricts to a
homeomorphism on each open cell. This means that $R(X)$ is an
$R$-complex with map $\phi_R \circ f:R(X)\rightarrow S_R$. We can
iterate this process to define $R^n(X)$ by setting $R^0 (X) =X$ (with
map $f:X\rightarrow S_R$) and $R^n(X)=R(R^{n-1}(X))$ (with map $\phi^n_R
\circ f:R^n(X)\rightarrow S_R$) if $n\geq 1$.

\begin{defi}
Let $R$ be a subdivision rule, and let $X$ be an $R$-complex. We will use $\Lambda_n$ to denote the union of all non-ideal tiles in the $n$th level of subdivision $R^n(X)$. The set $\mathop{\bigcap} \limits_{n} \Lambda_n \subseteq X$ is called the \textbf{limit set} and is denoted by $\Lambda=\Lambda(R,X)$. Its complement is called the \textbf{ideal set} and is denoted $\Omega=\Omega(R,X)$.
\end{defi}

\begin{defi} Let $R$ be a subdivision rule, and let $X$ be an $R$-complex. Let $\Gamma_n$ be the dual graph of $\Lambda_n$, i.e. a graph with

\begin{enumerate}
\item a vertex for each top-dimensional cell of $\Lambda_n$, and
\item an edge for each pair of top-dimensional cells that intersect in a codimension 1 subset.
\end{enumerate}

The \textbf{history graph} $\Gamma=\Gamma(R,X)$ consists of:
\begin{enumerate}
\item a single vertex $O$ called the \textbf{origin},
\item the disjoint union of the graphs $\{\Gamma_n\}$, whose edges are called \textbf{horizontal}, and
\item a collection of \textbf{vertical} edges which are induced by subdivision; i.e., if a vertex $v$ in $\Gamma_n$ corresponds to a $n$-cell $T$, we add an edge connecting $v$ to the vertices of $\Gamma_{n+1}$ corresponding to each of the $n$-cells contained in $R(T)$. We also connect the origin $O$ to every vertex of $\Gamma_0$.
\end{enumerate}
\end{defi}

(Note: In \cite{SubClass}, we used an alternative definition for history graph that had a vertex for every cell of any dimension in $\Lambda_n$, with edges being induced by inclusion. However, the two history graphs are quasi-isometric when any top-dimensional cells that intersect do so in a codimension 1 subset. The earliest form of a history graph, called the \textbf{history complex}, appears in \cite{cannon1999conformal}).

\subsection{Combinatorial subdivision rules}

First, we define labeled graphs and their morphisms. For background on labeled graphs, see \cite{gallian2009dynamic}. We use both edge and vertex labels.

\begin{defi}
A \textbf{(finitely) labeled graph} is a graph together with a map from the edges of the graph to a finite set of \textbf{edge labels}, and a map from the vertices of the graph to a finite set of \textbf{vertex labels}. For purposes of this article, we include unions of open edges as labeled graphs.
\end{defi}

\begin{defi}
A \textbf{labeled graph morphism} is a graph morphism between labeled graphs that preserves labels.
\end{defi}

\begin{defi}
The \textbf{open star} of a vertex is the union of a vertex with all the open edges that have that vertex as an endpoint.
\end{defi}

\begin{defi}
A finitely-labelled graph $\Xi$ is a \textbf{combinatorial subdivision rule} if it contains disjoint subgraphs $\Xi_n$ such that the following are satisfied:
\begin{enumerate}
\item $\Xi_0$ is a single vertex.
\item Every vertex is contained in some $\Xi_n$.
\item Every vertex $v$ of $\Xi_n$ for $n>0$ is connected to a unique vertex of $\Gamma_{n-1}$ called the \textbf{predecessor} of $v$. We define the predecessor of the unique vertex in $\Xi_0$ to be itself.
\item The open stars of any two vertices with the same label are labelled-graph isomorphic.
\item Condition 3 allows us to define a map $\pi:\Xi\rightarrow\Xi$, which is the graph morphism sending each vertex to its predecessor. We call the map $\pi$ the \textbf{predecessor map}. Then we require the preimages under $\pi$ of two edges with the same label to be labelled-graph isomorphic. A representative graph in such an isomorphism class is called an \textbf{edge subdivision}. Similarly, we require the preimage of two open stars of vertices with the same label to be labelled-graph isomorphic, and a representative graph in this isomorphism class is called a \textbf{vertex subdivision}.
\end{enumerate}
\end{defi}

\begin{lemma}
Each edge subdivision is a disjoint union of edges.
\end{lemma}
\begin{proof}
A labelled graph morphism is a homeomorphism when restricted to an open edgea, so each edge subdivision is a union of edges. Because no vertices are included, the edges are necessarily disjoint.
\end{proof}

\section{Main Theorems}

\begin{thm}\label{SubsComb}
Let $(R,X)$ be a finite subdivision pair. Then the history graph $\Gamma(R,X)$ is a combinatorial subdivision rule.
\end{thm}

\begin{proof}
Let $\Gamma$ be the history graph of a subdivision pair of dimension $n$. Items 1-3 in the definition of a combinatorial subdivision rule are automatically satisfied.

Vertex labels in $\Gamma$ correspond to non-ideal $n$-dimensional tile types, and edge labels correspond to $(n-1)$-dimensional tile types that are non-ideal and which are contained between two $n$-dimensional tiles in $X$. The tile type of an $n$-dimensional cell determines the tile type of its boundary, so the edge labels surrounding a given vertex label in $\Gamma$ are unique, satisfying item 4.

Finally, item 5 is satisfied by the nature of a subdivision rule: a subdivision rule acts locally, and always replaces a tile with a given type the exact same way.

\end{proof}

\begin{thm}\label{CombSubs} Let $\Xi$ be a combinatorial subdivision rule. Then there is a 3-dimensional subdivision pair $(R,X)$ such that the history graph $\Gamma(R,X)$ is graph isomorphic to $\Xi$.
\end{thm}

\begin{proof}
We need to construct the subdivision complex $S_R$, the related complex $R(S_R)$, and the subdivision map $\phi_R$ explicitly, as well as the $R$-complex $X$.

For each vertex label $v$, let $B(v)$ be a closed ball in $\mathbb{R}^3$. We place a cell structure on $B(v)$ such that:
\begin{enumerate}
\item there is one 3-cell in $B(v)$, and
\item the boundary sphere of $B(v)$ contains disjoint disks, one for each edge in the open star of a vertex with the label $v$. We give each disk the standard cell structure with one vertex, one edge, and one 2-cell. We consider the remainder of the sphere `ideal'.
\end{enumerate}

Let $Y$ be the quotient of the disjoint union of the $B(v)$ given by identifying boundary disks corresponding to edges with the same label. Thus, $Y$ has one disk for each edge label, via an orientation-preserving map. For each disk $D(e)$, let $D^\prime(e)$ be a new cell structure on $D(e)$ that contains disjoint sub-disks, one for each edge in the edge subdivision corresponding to $D(e)$.

The vertex subdivision corresponding to a vertex label is a union of open vertex stars. Given a label $v$, we construct a complex $N(v)$ taking the disjoint of copies of the $B(w)$, one for each vertex in the vertex subdivision of $v$, and identifying boundary disks that correspond to the same edge in the vertex subdivision, via an orientation-preserving map. Because each copy of a $B(w)$ deformation retracts onto compactification of the closed star of the corresponding vertex (here, the compactification merely adds endpoints onto the boundary edges), the whole complex $N(v)$ deformation retracts onto the compactification of the vertex subdivision. The unidentified boundary disks of $N(v)$ are called the \textbf{exterior disks} of $N(v)$. Each exterior disk of $N(v)$ corresponds to a boundary edge of the vertex subdivision, which in turn corresponds to an edge in the edge subdivision of one of the edges of the original vertex star.Thus, each exterior disk corresponds to a subdisk in some $D^\prime(e)$.

We now embed each $N(v)$ into $B(v) \subseteq Y$ so that:
\begin{enumerate}
\item the intersection of $N(v)$ with the boundary of $B(v)\subseteq Y$ is the union of the exterior disks of $N(v)$,
\item each exterior disk of $N(v)$ matches up with the appropriate subdisk of the appropriate $D'(e)$,
\item the closed complement $\overline{B(v)\setminus N(v)}$ is divided into 3-cells that are almost polyhedral (which we can do by triangulating and using barycentric subdivision twice, if necessary). We label this complement as ideal.
\end{enumerate}

This gives us a new cell structure on each $B(v)$, which we can call $B'(v)$, and thus a new cell structure on $Y$, which we can call $Y^\prime$; the two complexes $Y$ and $Y'$ have the same underlying topological space. We now let $I(v)$ be a cell complex isomorphic to $B(v)\setminus N(v)$. Note that the boundary of $I(v)$ can be partitioned into its intersection with the boundary of $B(v)$ (the \textbf{outer} portion of $\partial I(v)$) and its intersection with the boundary of $N(v)$ (the \textbf{inner} portion of $\partial I(v)$). Attach $I(v)$ to $Y$ by using the identity map on the outer portion and, on the inner boundary, mapping each part of $\partial N(v)$ to the $\partial B(v)$ it is a copy of.

Call this new complex $S_R$. If we replace the part corresponding to $Y$ with the $Y'$ structure, we get a new complex which we call $R(S_R)$.

There is a natural map from $R(S_R)$ to $S_R$, which is obtained by:
\begin{enumerate}
\item mapping each boundary sub-disk of the $N(v)$'s corresponding to an edge label $e$ to the disk $D(e)\subseteq S_R$ via an orientation-preserving map,
\item mapping each $I(v)$ to itself via the identity,
\item sending each closed 3-cell in $N(v)$ to the $B(v)$ it is a copy of, and
\item sending each complex $\overline{B(v)\setminus N(v)}$ to the $I(v)$ that is a copy of it.
\end{enumerate}

We call this map $\phi_R$, and, together with $S_R$ and $R(S_R)$, this forms a finite subdivision rule $R$ of dimension 3.

To create the cell complex $X$, recall that $\Xi_1$ is the set of all things of distance 1 from the origin in $\Xi$. Let $X$ consist of a copy of $B(v)$ for each open vertex star in $\Xi_1$ with label $v$, where we identify boundary disks of two $B(v)$'s that correspond to the same edge of $\Xi_1$.

Then the dual graph of $X$ is graph isomorphic to $\Xi_1$.

Let the structure map $f:X\rightarrow S_R$ be given by mapping each copy of a $B(v)$ to the corresponding $B(v)$ in $S_R$. Recall that the subdivision $R(X)$ is obtained by `pulling back' the cell structure of $R(S_R)$ via $f$. In this case, it replaces each copy of a $B(v)$ with $B(v)^\prime$. The non-ideal cells of $B(v)^\prime$ are the interior cells of $N(v)$ and its boundary disk; thus, the dual graph of $R(X)$ is obtained by replacing each open vertex star of the dual graph of $XS$ with its vertex subdivision, and each edge with its edge subdivision. Thus, the dual graph of $R(X)$ must be isomorphic to $\Xi_2$. By continuing this process, we see that $R^n(x)$ is dual to $\Xi_{n+1}$, and that the vertical edges of $\Gamma(R,X)$ connect a vertex to its predecessor under the subdivision.

This concludes the proof.
\end{proof}

\section{Applications}

\begin{thm}
Every hyperbolic group that is the fundamental group of a compact special cube complex has a Gromov boundary that is the quotient of a compact subset of $\mathbb{R}^3$ with connected preimages.
\end{thm}

\begin{proof}
In \cite{SubClass}, we showed that the limit set $\Lambda$ of a subdivision pair whose history graph is quasi-isometric to a hyperbolic group $G$ has a canonical quotient onto the Gromov boundary $\partial G$, with connected preimages.

In \cite{rushton2013subdivision}, we showed that every compact special cube complex has a fundamental group that is quasi-isometric to the history graph of some subdivision pair.

Combining these results with those of this paper, we see that every hyperbolic group that is the fundamental group of a compact special cube complex has a compact limit set $\Lambda$ that is a subset of $\mathbb{R}^3$, and which quotients onto the Gromov boundary with connected preimage.
\end{proof}

Note that the Hahn-Mazurkiewicz theorem states that every locally-connected continuum is a quotient of the unit interval \cite{willard2004general}. An example of this is the famous Peano space-filling curve. However, in most of these examples, the preimages are not connected.

\bibliographystyle{plain} \bibliography{Hyperbolicbib}

\end{document}